\documentclass[12pt,a4paper]{article}
\usepackage{amsmath,amssymb,centernot, amsthm}

\newtheorem{thm}{Theorem}[section]
\newtheorem{lem}[thm]{Lemma}

\newtheorem{cor}[thm]{Corollary}

\newtheorem{problem}[thm]{Problem}
\newtheorem{re}[thm]{Result}
\newtheorem{question}[thm]{Question}

\newtheorem{example}[thm]{Example}

\newtheorem{remark}[thm]{Remark}

\newcommand{\Z}{\mathbb{Z}}
\newcommand{\F}{\mathbb{F}}

\newcommand{\Q}{\mathbb{Q}}
\newcommand{\R}{\mathbb{R}}

\begin{document}
\title{Unique Differences in Symmetric Subsets of $\F_p$}
\author{Tai Do Duc\\ Division of Mathematical Sciences\\
School of Physical \& Mathematical Sciences\\
Nanyang Technological University\\
Singapore 637371\\
Republic of Singapore\\[5mm]
Bernhard Schmidt\\ Division of Mathematical Sciences\\
School of Physical \& Mathematical Sciences\\
Nanyang Technological University\\
Singapore 637371\\
Republic of Singapore}
\date{}

\maketitle

 \begin{abstract}
Let $p$ be a prime and let $A$ be a subset of $\F_p$ with $A=-A$
and $|A\setminus\{0\}| \le 2\log_3(p)$. Then there is an element of
$\F_p$ which has a unique representation as a difference
of two elements of $A$.
 \end{abstract}


\section{Introduction}
Let $p$ be a prime and let $\F_p$ denote the field with $p$ elements.
Let $A$ be a nonempty subset of $\F_p$.
We say that $A$ has a \textbf{unique difference}
if there is $x \in \F_p$ such that
there is exactly one ordered pair $(a,b)$, $a,b\in A$, with $x=a-b$.
Unique sums are defined similarly.
A subset $B$ of $\F_p$ is called \textbf{symmetric} if $B=-B$.

\medskip

According to \cite{str}, the following problem was first proposed by W.\ Feit.

\medskip

\begin{problem}
Given a prime $p$,
what is the largest number $f(p)$ such that every subset of $\F_p$
with at most $f(p)$ elements has a unique difference{\rm ?}
\end{problem}

\medskip

Let $\log_n(\cdot)$ denote the logarithm with base $n$.
Straus \cite{str} showed $f(p)\ge 1+\log_4(p-1)$.
This result was improved by Browkin, Divis, and Schinzel \cite{bro}
who obtained
\begin{equation} \label{browkin1}
f(p)\ge \log_2p,
\end{equation} which is the best known
lower bound for $f(p)$. It is not known whether this bound is
asymptotically sharp.

\medskip

In  \cite[Thm.\ 2]{str}, subsets of $\F_p$ were constructed which do
not have unique differences and are of cardinality $(2+o(1))\log_3(p)$.
These sets are symmetric.  Thus \cite[Thm.\ 2]{str} implies
\begin{equation} \label{straus}
g(p)\le (2+o(1))\log_3(p)
\end{equation} where
$g(p)$ denotes the largest number such that every \textit{symmetric}
subset of $\F_p$ with at most $g(p)$ elements has a unique difference.

\medskip

Note that a symmetric subset of $\F_p$ has a unique difference if
and only if it has a unique sum. Thus we could as well formulate
the results of this paper in terms of unique sums.

\medskip

The above-mentioned result of Browkin, Divis, and Schinzel
implies
\begin{equation} \label{browkin}
g(p)\ge \log_2p.
\end{equation}
The following theorem is the main result of this paper. It implies $g(p)\ge 2\log_3(p)$,
which is a substantial improvement upon (\ref{browkin}).

\begin{thm} \label{main}
Let $p\ge 5$ be a prime and let $A$ be a symmetric subset
of $\F_p$ with $|A\setminus\{0\}| \le 2\log_3(p)$.
Then $A$ has a unique difference.
\end{thm}

In view of Straus' result (\ref{straus}), Theorem \ref{main} is sharp in the following sense.

\begin{cor}\label{sharpness}
We have $g(p)\ge 2\log_3(p)$ for every prime $p$.
Moreover, for every $\varepsilon>0$, there exists a constant $C(\varepsilon)$
such that
\begin{equation} \label{sharpness1}
 g(p) \le (2+\varepsilon)\log_3(p)
\end{equation}
for every prime $p>C(\varepsilon)$.
\end{cor}

Results like (\ref{browkin1}) and Theorem \ref{main} have applications in various areas,
see \cite{bro, leu, lev, lox1, ned1, ned2, ned3}, for instance.
In the last section, we present a new application to cyclotomic integers $X$
for which $|X|^2$ is an integer.


\section{Preliminaries}
In this section, we state some well known results which will be needed later.
We include  proofs for the convenience of the reader.
For a ring $R$, let $M_{m,n}(R)$ denote the set of $m\times n$ matrices with entries from $R$.
The Euclidean norm of $x\in \R^n$ is denoted by $||x||$.
\begin{re} \label{QR}
Let $m\le n$ and let  $A\in M_{m,n}(\R)$ with rows $r_1,\ldots,r_m$.
Set $d(1)=||r_1||$. For $2\le j\le m$,  let $d(j)$ be the distance of $r_j$ from the subspace
of $\R^n$ spanned by $r_1,\ldots,r_{j-1}
$.
We have
\begin{equation} \label{QR1}
\det(AA^T)=\prod_{j=1}^{m}d(j)^2.
\end{equation}
\end{re}
\begin{proof}
If ${\rm rank}_{\R}(A)<m$, then left hand and right side of (\ref{QR1}) are both zero. Hence we may assume ${\rm rank}_{\R}(A)=m$.
By Gram-Schmidt orthogonalization, there is a nonsingular lower triangular matrix $L\in M_{m\times m}$
with diagonal entries $d(j)^{-1}$, $j=1,\ldots,m$, such that the rows of $Q=LA$
are an orthonormal basis of the subspace of $\R^n$ spanned by $r_1,\ldots,r_m$.
Thus $QQ^T=I_m$ where $I_m$ denotes the $m\times m$ identity matrix.
We conclude $AA^T = L^{-1}QQ^T(L^{-1})^T=L^{-1}(L^{-1})^T$. Hence
$$\det(AA^T)=(\det(L^{-1}))^2 = \prod_{j=1}^{m}d(j)^2.$$
\end{proof}

\begin{re} \label{volume}
Let $A\in M_{u,n}(\R)$, $B\in M_{w,n}(\R)$, and $C={A \choose B}$. Then
$$\det(CC^T) \le \det(AA^T)\det(BB^T).$$
\end{re}
\begin{proof} Let $r_1,\ldots,r_u$ be the rows of $A$ and let
$r_{u+1},\ldots,r_{u+w}$ be the rows of $B$.
For $1\le i <j \le u+w$, let $d(i,j)$ denote the distance of $r_j$
from the subspace of $\R^n$ generated by $r_i,\ldots,r_{j-1}$.
Furthermore, set $d(i,i)=||r_i||$ for all $i$.
By Result \ref{QR}, we have
\begin{align*}
\det(AA^T) & =\prod_{j=1}^{u}d(1,j)^2,\\
\det(BB^T) &=\prod_{j=u+1}^{u+w}d(u+1,j)^2,\\
\det(CC^T)&=\prod_{j=1}^{u+w}d(1,j)^2.
\end{align*}
Moreover, $d(1,j)\le d(u+1,j)$ for  $j\ge u+1$ by the definition of the $d(i,j)$'s. Hence
\begin{align*}
\det(CC^T) & =  \prod_{j=1}^{u+w}d(1,j)^2\\
& \le  \prod_{j=1}^{u}d(1,j)^2\prod_{j=u+1}^{u+w}d(u+1,j)^2\\
& =  \det(AA^T)\det(BB^T).
\end{align*}
\end{proof}

A repeated application of Result \ref{volume} gives the following.
\begin{cor} \label{volumecor}
Let $A_i\in M_{u_i,n}(\R)$, $i=1,\ldots,k$, and
$$C=\begin{pmatrix} A_1\\ A_2\\ \vdots\\A_k\end{pmatrix}.$$
 Then
$$\det(CC^T) \le \prod_{i=1}^k\det(A_iA_i^T).$$
\end{cor}


\section{Set-up} \label{setup}
In this section, we introduce some notation and assumptions which will
be implicitly assumed in the rest of the paper.
Let $p$ be an odd prime and suppose that $A$ is a
symmetric subset of $\F_p$ which has no unique difference.

\medskip

Recall that we assume $p\ge 5$.
Write $A=\left\{\pm a_1,...,\pm a_n\right\}$ such that $a_j\neq \pm a_i$ for $i\neq j$.
We set $a_n=0$ if $0\in A$. Let $m=\lfloor |A|/2 \rfloor$.
Note that $m=n$ if $0 \not\in A$ and $m=n-1$ if $0 \in A$.

\medskip

Next, we show that we may assume $|A|\ge 4$.
If $|A|=1$, then $A$ certainly has a unique difference.
If $2\le |A|\le 3$, then $A=\{\pm a_1\}$ or $A=\{0,\pm a_1\}$ for some $a_1\neq 0$.
Then $a_1-(-a_1)=2a_1$ is a unique difference in $A$.
Hence we can indeed assume $|A|\ge 4$.

\medskip

We now set up a linear system arising from $A$
and derive some useful properties of its coefficient matrix.
Let $i$ be arbitrary with $1\le i \le m$.
As $2a_i=a_i-(-a_i)$ is not a unique difference in $A$ and $2a_i\neq \pm a_i-a_i$, there exists an ordered pair
$(\sigma(i), \tau(i))\neq (i,i)$ with
\begin{equation}\label{differences}
2a_i \pm a_{\sigma(i)} \pm a_{\tau(i)}=0.
\end{equation}
Here ``$\pm a_{\sigma(i)} \pm a_{\tau(i)}$'' means that \textit{any} combination
of signs is possible including $a_{\sigma(i)} - a_{\tau(i)}$ and
$-a_{\sigma(i)} + a_{\tau(i)}$. We use this convention throughout the rest
of the paper.

\medskip

We consider the homogeneous linear system corresponding to these equations:
\begin{equation}\label{system}
2x_i \pm x_{\sigma(i)} \pm x_{\tau(i)}=0,\ i=1,\ldots,m.
\end{equation}
Here we use the convention $x_n=0$ if $0\in A$ (and thus $a_n=0$).

\medskip

Note that the coefficient vectors corresponding to the system
(\ref{system}) all have at most $3$ nonzero entries.
The nonzero coefficients, however, are not necessarily $2,\pm 1$,
since $i,\sigma(i),\tau(i)$ are not necessarily distinct.
We now determine exactly which coefficient vectors can occur.

\medskip

\noindent
{\bf Case 1} $\sigma(i)=i$ or $\tau(i)=i$.  By symmetry, we can assume $\tau(i)=i$.
Thus $2a_i \pm a_{\sigma(i)} \pm a_{i}=0$. As $(\sigma(i), \tau(i))\neq (i,i)$,
we have $\sigma(i)\neq i$. If  $2a_i \pm a_{\sigma(i)} - a_{i}=0$,
then $a_i\pm a_{\sigma(i)}=0$, contradicting the assumption $a_i\neq \pm a_j$ for $i\neq j$.
Hence $2a_i \pm a_{\sigma(i)} + a_{i}=3a_i\pm a_{\sigma(i)}=0$.
Recall that we assume $p\ge 5$.
If $a_{\sigma(i)}=0$, then $3a_i=0$ and thus $a_i=0$, contradicting the assumptions.
Thus $a_{\sigma(i)}\neq 0$, i.e., $\sigma(i)\le m$.

\medskip

Hence (\ref{differences}) can be written as
\begin{equation}\label{type1}
3x_i\pm x_{\sigma(i)}=0
\end{equation}
with $\sigma(i)\neq i$ and $\sigma(i)\le m$.
We call (\ref{type1}) an equation of \textbf{type 1}.

\medskip

\noindent
{\bf Case 2} $\sigma(i)\neq i$ and $\tau(i)\neq i$.
If $\tau(i)=\sigma(i)$, then (\ref{differences})
implies $a_i=0$ or $a_i=\pm a_{\sigma(i)}$, contradicting the assumptions.
Thus $\tau(i)\neq \sigma(i)$. Hence $i, \sigma(i), \tau(i)$ are pairwise
distinct.

\medskip

Recall that $x_n=0$ if $0\in A$.  Hence,  if $0\in A$
and $n\in \{\sigma(i),\tau(i)\}$, then
one of the variables $x_{\sigma(i)}, x_{\tau(i)}$
occurs with coefficient zero in (\ref{differences}).
By symmetry, we can assume that $x_{\tau(i)}$
occurs with coefficient zero in this case.

\medskip

Hence, in Case 2, we can write (\ref{differences}) as
\begin{equation}\label{type2}
\begin{split}
2x_i \pm x_{\sigma(i)} =0 & \text{\ \ if $0\in A$ and $\tau(i)=n$},\\
2x_i \pm x_{\sigma(i)} \pm x_{\tau(i)}=0 & \text{\ \ otherwise,}
\end{split}
\end{equation}
where $i, \sigma(i), \tau(i)$ are pairwise distinct.
In both cases, we call (\ref{type2}) an equation of \textbf{type 2}.

\medskip

In summary, as Case 1 and Case 2 cover all possible cases,
for every $i\in \{1,\ldots,m\}$, one of the following equations
is contained in the linear system (\ref{system}).
\begin{equation*}
\begin{split}
3x_i\pm x_{\sigma(i)}=0 & \text{\ \ (type 1)},\\
2x_i \pm x_{\sigma(i)} =0 & \text{\ \ (type 2)},\\
2x_i \pm x_{\sigma(i)} \pm x_{\tau(i)}=0 & \text{\ \ (type 2).}
\end{split}
\end{equation*}
Furthermore, the following hold.
\begin{itemize}
\item $\sigma(i)\le m$ and $\tau(i)\le m$,
\item  $i,\sigma(i),\tau(i)$ are pairwise distinct.
\end{itemize}
Of course, the statements involving $\tau(i)$ only apply if the equation $2x_i \pm x_{\sigma(i)} \pm x_{\tau(i)}=0$
is contained (\ref{system}).

\medskip

We use a similar terminology for the coefficient vectors
of the system (\ref{system}):  We say a coefficient vector  is of \textbf{type 1} if
it has exactly one entry $3$, exactly one entry $\pm 1$, and all its
remaining entries are zero.
A coefficient vector  is of \textbf{type 2} if it has
exactly one entry $2$, at most two entries $\pm 1$,
and all its remaining entries are zero.


\section{A Congruence for Minors of the Coefficient Matrix}
Let $M$ be the coefficient matrix of the linear system (\ref{system}).
Recall that all entries of $M$ are from $\{0,\pm 1,2,3\}$.
Note that $M$ can be considered as a matrix with rational entries as well as a matrix
with entries from $\F_p$.
In the following, we make use of both interpretations.
Let $r$ be the rank of the coefficient matrix $M$ over $\Q$.
We now prove that all $r\times r$-minors of $M$ are divisible by $p$.
This result is useful, as it can be combined with estimates for minors of $M$ which we get from Result \ref{QR}.
 \begin{thm}\label{snf}
Suppose that $A=\{\pm a_1,...,\pm a_n \}\subset \F_p$ does not have a unique difference
and let $M$ be the coefficient matrix of the corresponding linear system $(\ref{system})$.
Write $r={\rm rank}_{\Q}(M)$. Every $r\times r$-minor of $M$ is divisible by $p$.
\end{thm}
\begin{proof}
Recall that $M$ is an $m\times m$-matrix and that $M$ has entries from $\{0,\pm 1,2,3\}$
only. In this proof, we consider the entries of $M$ as \textit{integers}
(not as elements of $\F_p$), and we will work with the Smith Normal Form of $M$
over the integers.

\medskip

Let $D=\text{diag}(d_1,...,d_r,0,...,0)$ be the Smith Normal Form of $M$.
Note that $d_1,\ldots,d_r$ are integers such that $d_i$ divides $d_{i+1}$ for $i=1,\ldots,r-1$.
Furthermore, there are $B,C\in M_{m,m}(\Z)$ with $\det(B)=\pm 1$, $\det(C)=\pm 1$,
and $M=BDC$. It is important for this proof to note that
$B$, $C$, $D$ are matrices with \textit{integer} entries and the equation $M=BDC$ holds
over $\Z$ (not only over $\F_p$).
Furthermore, note that all entries of $C^{-1}$ are integers, as $\det(C)=\pm 1$.

\medskip

Write $x=(a_1,\ldots,a_m)^T$. Note $x\in \F_p^m$.
We have $x\neq 0$, as $|A|>1$.
Recall that $Mx=0$ over $\F_p$ by (\ref{system}).

\medskip

Suppose that $p$ does not divide $d_r$. Then $p$ does not divide any of the integers $d_1,\ldots,d_r$.
Write $Cx=(b_1,\ldots,b_m)^T$
where $b_1,\ldots,b_m\in \F_p$.
Note that $DCx=0$ over $\F_p$ if and only if $b_1=\cdots= b_r=0$, since
we assume that $d_1,\ldots,d_r$ are not divisible by $p$.
Hence $DCx=0$ if and only if $x=C^{-1}(0,\ldots,0,b_{r+1},\ldots,b_{m})^T$ for some $b_i\in \F_p$.
As $BDCx= Mx=0$ over $\F_p$, we indeed have $DCx=0$ over $\F_p$.
Thus
\begin{equation} \label{snf1}
x=C^{-1}b \text{ with } b=(0,\ldots,0,b_{r+1},\ldots,b_{m})^T\in \F_p^m.
\end{equation}
For convenience, we set $b_i=0$ for $i=1,\ldots,r$.

\medskip

We now switch from equations over $\F_p$ to equations over $\Z$.
We first define a vector $c$ with \textit{integer entries}
which represent the residue classes $b_1,\ldots,b_m$ mod $p$.
Formally, let $c_1,\ldots,c_m$ be the unique integers
with $0\le c_i \le p-1$ and
\begin{equation}  \label{snf1a}
b_i = c_i+p\Z/\Z
\end{equation}
 (here we use the standard notation $\F_p=\{k+p\Z/\Z: k=0,\ldots,p-1\}$).
Note that $c_1=\cdots=c_r=0$, as $b_1=\cdots = b_r=0$
(again, note that $c_1=\cdots=c_r=0$ are equations over $\Z$,
not only over $\F_p$).

\medskip

Define $y=(y_1,\ldots,y_m)\in \Z^m$ by
\begin{equation} \label{snf2}
y=C^{-1}(0,\ldots,0,c_{r+1},\ldots,c_{m})^T
\end{equation}
(recall that the $c_i$'s are considered as integers, not as elements of $\F_p$).
Then we have
\begin{equation} \label{snf3}
My=BDCy=B\,\text{diag}(d_1,...,d_r,0,...,0)(0,\ldots,0,c_{r+1},\ldots,c_{m})^T=0
\end{equation}
over $\Z$ (note that all entries of matrices and vectors
occurring in (\ref{snf3}) are
considered as integers and to derive (\ref{snf3}),
we need the fact that the equation $M=BDC$ holds over $\Z$,
not only over $\F_p$).

\medskip

Let $\Gamma_j$ be the $j$th row of $C^{-1}$, $j=1,\ldots,m$.
By (\ref{snf1}) and (\ref{snf2}), we have
\begin{equation} \label{snf4}
a_j = \Gamma_jb \text{ and } y_j = \Gamma_jc
\end{equation}
where $b=(0,\ldots,0,b_{r+1},\ldots,b_{m})^T$
and $c=(0,\ldots,0,c_{r+1},\ldots,c_{m})^T$.
Note that (\ref{snf1a}) and (\ref{snf4}) imply
\begin{equation}  \label{snf5}
a_j = y_j+p\Z/\Z,\ i=1,\ldots,m.
\end{equation}

\medskip

Recall that $M$ is the coefficient matrix of the linear system (\ref{system})
and that $My=0$ over $\Z$ by (\ref{snf3}).
This implies that $\{\pm y_1,\ldots,\pm y_m\}$ does not have a unique difference (where the
differences are taken in $\Z$). Recall that $a_i\neq \pm a_j$ for $i\neq j$ by assumption.
Thus (\ref{snf5}) implies  $y_i\neq \pm y_j$ for $i\neq j$. Hence there is $k\in\{1,\ldots,m\}$
such that $|y_k|>|y_i|$ for all $i\neq k$.
This implies that $2y_k = y_k-(-y_k)$ is a unique difference of $\{\pm y_1,\ldots,\pm y_m\}$, a contradiction.
We conclude that $p$ divides $d_r$.

\medskip

From the theory of the Smith Normal Form  (see \cite[p.\ 41]{nor}, for instance), it is well known that
$d_r$ divides the greatest common divisor of all $r\times r$-minors of $M$.
 \end{proof}


\section{Equations of Type 1}
Equations of type 1 play a critical role in the proof of Theorem \ref{main},
as the Euclidean norm of their coefficient vectors is the largest among
the equations occurring in the linear system (\ref{system}). In this section,
we study the structure of the set of equations of type 1 contained
in (\ref{system}).
Recall that equations of type 1 have the form
$$
3x_i\pm x_{\sigma(i)}=0
$$
where $1\le i\le m$, $\sigma(i)\neq i$, and $\sigma(i)\le m$.

\begin{lem} \label{graph}
Suppose $3^{\lfloor |A|/2 \rfloor}\le p$.
Let $I$ be a subset of $\{1,\ldots,m\}$ such
that $3x_i\pm x_{\sigma(i)}=0$, $i\in I$,  are equations of type $1$
contained in $(\ref{system})$.
Let $G$ be the directed graph with vertex set $I\cup \{\sigma(i):i\in I\}$ and  edge set
$E=\{(i,\sigma(i)): i\in I\}$. Then $E$ can  be decomposed into
directed paths which are pairwise vertex disjoint.
\end{lem}

\begin{proof}
We show\\
\mbox{\ \ \ } (i) every vertex of $G$ has outdegree at most $1$,\\
\mbox{\ \ \ }(ii) every vertex of $G$ has indegree at most $1$,\\
\mbox{\ \ \ }(iii) $G$ does not contain a directed cycle.

\medskip

Note that (i-iii) imply that $G$ indeed can be decomposed into
directed paths which are pairwise vertex disjoint.

\medskip

First of all, there is at most one edge $(i,\sigma(i))$ for every vertex $i$. This means that the outdegree
of every vertex in $G$ is at most one.

\medskip

Suppose that the indegree of a vertex $i$ is at least $2$. Then there are distinct vertices $j,k$
with $\sigma(j)=\sigma(k)=i$. By definition, this implies
$3a_j\pm a_i=0$ and $3a_k\pm a_i=0$ and thus $3(a_j-a_k)=0$ or $3(a_j+a_k)=0$.
As $p>3$, we conclude $a_j=\pm a_k$ which contradicts the assumption $a_j\neq \pm a_k$ for $j\neq k$.
This shows that all vertices of $G$ have indegree at most $1$.

\medskip

Now suppose that $G$ contains a directed cycle. Then there are vertices $v_1,v_2,\ldots,v_k$ with $v_k=v_1$
such that $(v_i,v_{i+1})\in E$ for $i=1,\ldots,k-1$, i.e., $v_{i+1}=\sigma(v_i)$ for $i=1,\ldots,k-1$.
By definition, this implies $3a_{v_i}= \pm a_{v_{i+1}}$ for $i=1,\ldots,k-1$.
Hence
$$a_{v_1}=a_{v_k}=\pm 3a_{v_{k-1}}=\pm 9a_{v_{k-2}}=\cdots =\pm 3^{k-1}a_{v_1}.$$
Thus $(\pm 3^{k-1}-1)a_{v_1}=0$.
Note that $k\le m+1$, as the cycle contains every vertex $\neq v_1$ at most once.
Moreover, $p\ge 3^m= 3^{\lfloor |A|/2 \rfloor}$ by assumption and thus $p\ge 3^m+2$, as $p>3$ is a prime.
We conclude $3^{k-1}+1\le 3^m+1<p$.
Hence $(\pm 3^{k-1}-1)a_{v_1}=0$ implies $a_{v_1}=0$, contradicting the assumption
$a_i\neq 0$ for $i=1,\ldots,m$. This shows that $G$ does not contain a directed cycle.

\medskip

In summary, we have shown that (i)-(iii) hold, and this completes the proof.
\end{proof}

\begin{remark}\rm
Note that Lemma \ref{graph} implies that $G$ does not contain \textit{any} cycle, directed or undirected.
\end{remark}

Next, we compute determinants arising from matrices whose rows are coefficient vectors
of equations of type 1.

\begin{lem} \label{type1det}
Let $G$ be  the graph defined in Lemma $\ref{graph}$.
Suppose $J$ is a subset of $\{1,\ldots,m\}$ such
that  $\{(i,\sigma(i)): i\in J\}$ is a directed path in $G$
and let $B$ be the coefficient matrix of the corresponding equations $3x_i\pm x_{\sigma(i)}=0$, $i\in J$.
Then
\begin{equation} \label{type1det1}
\det(BB^T) = \frac{1}{8}\left(-1 + 9^{|J|+1}\right).
\end{equation}
\end{lem}
\begin{proof}
Write $v=|J|$. By relabeling vertices of $G$, if necessary, we may assume that
the directed path consists of the edges  $(i,i+1)$, $i=1,\ldots,v$.
The corresponding equations are $3x_i\pm x_{i+1}=0$, $i=1,\ldots,v$. Hence
$B$ is a $v\times m$-matrix of the form
$$
\begin{pmatrix}
3 & \pm 1&0& \cdots &0&0   &0&\cdots&0\\
0& 3 & \pm 1& \cdots &0&0  &0&\cdots&0\\
\vdots&\vdots&\vdots& \ddots& \vdots&\vdots  &0&\cdots&0\\
0&0&0&  \cdots &3&\pm 1 &0&\cdots&0\\
\end{pmatrix}.
$$
Thus
$$
BB^T=
\begin{pmatrix}
10 & 3\delta_1 &0& \cdots &0&0&0\\
3\delta_1& 10 & 3\delta_2& \cdots &0&0&0\\
0 & 3\delta_2& 10 & 3\delta_3& \cdots &0&0\\
\vdots&\vdots&\ddots& \ddots& \ddots&\vdots\\
0& \cdots & 0& 3\delta_{v-3}&10&3\delta_{v-2}&0\\
0& \cdots & 0& 0&3\delta_{v-2}&10&3\delta_{v-1}\\
0& \cdots & 0 & 0 &0&3\delta_{v-1}&10
\end{pmatrix}
$$
with $\delta_i=\pm 1$, $i=1,\ldots,v-1$. We now prove (\ref{type1det1})
by induction on $v$. It is straightforward to check that  (\ref{type1det1}) holds for $v=1$.
Suppose $v\ge 2$. Using Laplace expansion with respect to the last column
of $BB^T$ and the inductive hypothesis, we find
$$\det(BB^T)=10\left(\frac{1}{8}\left(-1 + 9^{v}\right)\right)
           -(3\delta_{v-1})^2\left(\frac{1}{8}\left(-1 + 9^{v-1}\right)\right).$$
Note $(3\delta_{v-1})^2=9$.
Hence
$$\det(BB^T)
=\frac{1}{8}\left(-10+9 +(10-1)9^{v}\right)
= \frac{1}{8}\left(-1+9^{v+1}\right).$$
\end{proof}

\begin{cor} \label{type1bound}
Suppose $3^{\lfloor |A|/2 \rfloor}\le p$.
 Let $I$ be a subset of  $\{1,\ldots,m\}$ such
that $3x_i\pm x_{\sigma(i)}=0$, $i\in I$,  are  equations of type $1$.
Let $G$ be the directed graph defined in Lemma $\ref{type1}$
and let $P_1,\ldots,P_t$ be vertex disjoint directed paths in $G$ with $E=\displaystyle\bigcup_{j=1}^t P_j$.
Let $B$ be the matrix whose rows are the coefficient vectors
of the equations $3x_i\pm x_{\sigma(i)}=0$, $i\in I$.
Then
$$\det(BB^T) < 11^{t}9^{|I|-t}.$$
\end{cor}
\begin{proof}
For  $j=1,\ldots,t$,
let $B_j$ be the matrix whose rows are the coefficient vectors of the equations
$3x_i\pm x_{\sigma(i)}=0$, $(i,\sigma(i))\in P_j$. By Lemma \ref{type1det},
we have
$$\det(B_jB_j^T)= \frac{1}{8}\left(-1 + 9^{|P_j|+1}\right).$$
Note that
$$\frac{1}{8}\left(-1 + 9^{x+1}\right)<11\cdot 9^{x-1}$$
for all $x\ge 1$.  Using Corollary \ref{volumecor} and $\sum_{j=1}^t|P_j|=|I|$, we conclude
\begin{align*}
\det(BB^T)  & \le  \prod_{j=1}^t \det(B_jB_j^T)\\
 & =    \prod_{j=1}^t\left( \frac{1}{8}\left(-1 + 9^{|P_j|+1}\right)\right)\\
& <  \prod_{i=1}^t \left(11\cdot 9^{|P_j|-1}\right)\\
& =  11^t9^{|I|-t}.
\end{align*}
\end{proof}


\section{Proof of Theorem \ref{main} }
Let $p$ be prime and let $A$ be a symmetric subset of $\F_p$ with
\begin{equation} \label{main1}
|A\setminus\{0\}| \le 2\log_3(p).
\end{equation}
We have to prove that $A$ has a unique difference. As shown in Section \ref{setup},
we may assume $p\ge 5$.

\medskip

Suppose $A$ does not have a unique difference.
Write $A=\{\pm a_1,\ldots,\pm a_n\}$ as in Section \ref{setup} where $a_n=0$ if $0\in A$.
Recall $m=\lfloor A/2 \rfloor$.
Note $|A\setminus\{0\}|=2\lfloor A/2 \rfloor$. Hence $3^{\lfloor A/2 \rfloor} \le p$
by (\ref{main1}) and thus
\begin{equation} \label{main2}
3^{\lfloor A/2 \rfloor} < p,
\end{equation}
as $p\neq 3$.

\medskip

Let $M$ be the coefficient matrix of the linear system (\ref{system}) and write
$r={\rm rank}_{\Q}(M)$. Note $r\le m$. Let $N$ be a nonsingular $r\times r$ submatrix of $M$.

\medskip

We claim that
\begin{equation} \label{main-1}
|\det(N)|\le 3^r.
\end{equation}
We say that a row vector is of \textbf{type 3} if it has exactly one entry $3$
and all its remaining entries are zero.
Note that $N$ may contain rows of type 3, as some rows of $M$ of type 1 may
turn into type 3 when columns of $M$ are deleted.  If $N$ contains a row of type 3, then,
by Laplace expansion,
$|\det(N)|=3|\det(N_1)|$ where $N_1$ is a $(r-1)\times (r-1)$
submatrix of $M$. Repeating this process,  if necessary, we either get $|\det(N)|=3^r$
or obtain a
$d\times d$ submatrix $N_2$ of $M$ with
\begin{equation} \label{main0}
|\det(N)|=3^{r-d}|\det(N_2)|
\end{equation}
such that $N_2$ does not contain any row of type 3.
If $|\det(N)|=3^r$, then (\ref{main-1}) holds.
Thus we may assume that (\ref{main0}) holds.

\medskip

As $N_2$ is a submatrix of $M$, all rows of $N_2$ which are not of type 1 have 
at most three nonzero entries $\pm 1$, $2$, at most one of which
is $2$. Hence  every row of $N_2$ which
is not of type 1 has Euclidean norm  at most $\sqrt{6}$.

\medskip

Swapping rows, if necessary, we can write
$$N_2={B \choose C}$$
where $B$ consists of the rows of $N_2$ of type $1$
and $C$ of the remaining rows of $N_2$.
Let $I$ be the subset of $\{1,\ldots,m\}$ such that
$3x_i\pm x_{\sigma(i)}=0$, $i\in I$, are
the equations corresponding to the rows of $B$.
Note that $|I|$ is the number of rows of $B$ and $d-|I|$
is the number of rows of $C$.

\medskip

Let $G$ be the directed graph with vertex set $V=I\cup \{\sigma(i):i\in I\}$ and  edge set
$E=\{(i,\sigma(i)): i\in I\}$. Note that $|V|\le d$,
as every element of $V$ is an index of a column of $N_2$.

\medskip

By Lemma \ref{graph}, there is a decomposition of $E$ into pairwise
vertex disjoint directed paths.
Note that these directed paths are connected components of $G$.
Let $t$ be the number of paths in the decomposition.
As $G$ has at most $d$ vertices and contains no cycles, it
has at most $|V|-|E|$ connected components. Hence
\begin{equation} \label{main4}
t\le |V|-|E|\le d-|I|.
\end{equation}
By Corollary \ref{type1bound}, we have
\begin{equation} \label{main5}
\det(BB^T) < 11^{t}9^{|I|-t}.
\end{equation}
From (\ref{main4}) and (\ref{main5}), we get
\begin{equation} \label{main5a}
\det(BB^T) < 11^{d-|I|}9^{|I|-(d-|I|)}=11^{d-|I|}9^{2|I|-d}.
\end{equation}
As the rows of $C$ all have Euclidean norm at most $\sqrt{6}$,
we have
\begin{equation} \label{main6}
\det(CC^T) \le 6^{d-|I|}
\end{equation}
by Result \ref{QR}. From Result \ref{volume}, we get
\begin{equation} \label{main7}
\det(N_2N_2^T) \le \det(BB^T)\det(CC^T).
\end{equation}
Putting (\ref{main5a}-\ref{main7}) together, we find
\begin{align*}
\det(N_2N_2^T) & < 11^{d-|I|}9^{2|I|-d}6^{d-|I|}\\
& = 66^{d-|I|}9^{2|I|-d}\\
& \le 81^{d-|I|}9^{2|I|-d}\\
& = 9^{d}.
\end{align*}
Hence $|\det(N_2)|\le 3^d$ and thus $|\det(N)|\le 3^r$
by (\ref{main0}). This proves (\ref{main-1}).

\medskip

Finally, recall that $\det(N)$ is a nonzero $r\times r$-minor
of $M$. Hence $|\det(N)|\ge p$ by Theorem \ref{snf}.
But $r\le \lfloor |A|/2 \rfloor$ and thus $|\det(N)|\le 3^r<p$ by (\ref{main2}),
a contradiction. This completes the proof of Theorem \ref{main}.
\hfill $\Box$


\section{Application to Weil Numbers}
Let $m,n$ be positive integers.
Write $\zeta_m = \exp(2\pi i/m)$. An {\boldmath $n$}\textbf{-Weil number}
in $\Z[\zeta_m]$ is an element $Y$ of  $\Z[\zeta_m]$ with $|Y|^2=n$.
In this section, we show that
Theorem \ref{main} implies that under certain conditions Weil numbers in $\Z[\zeta_m]$ necessarily
are contained in proper subfields of $\Q(\zeta_m)$.
This result is a partial improvement of the ``field descent'' introduced in \cite{sc}
and is relevant for the study of difference sets and related objects.
We will assume basic algebraic number theory in this sections, as treated in \cite{ire}, for instance.

\medskip
For a finite group $G$ and a ring $R$, let $R[G]$ denote the group ring of $G$ over $R$.
Every element $B$ of $R[G]$ can be written as $B=\sum_{g\in G} r_gg$ with $r_g\in R$.
The $r_g$'s are called the \textbf{coefficients} of $B$.
We write $B^{(-1)}=\sum_{g\in G} \overline{r_g}g^{-1}$
where $\overline{r_g }$ is the complex conjugate of $r_g$.

\medskip

 For $Y\in \Z[\zeta_m]$, let
$${\cal M}(Y) = \frac{1}{\varphi(m)} \sum_{\sigma\in {\rm Gal}(\Q(\zeta_m)/\Q)} (Y\overline{Y})^{\sigma},$$
where $\varphi$ denote the Euler totient function.
Note
\begin{equation} \label{geoarith}
{\cal M}(Y)\ge 1
\end{equation} for $Y\neq 0$ by the inequality of geometric
and arithmetic means, since $\prod {(Y\overline{Y})^{\sigma}}$ is the norm
of an algebraic integer and thus  $\prod {(Y\overline{Y})^{\sigma}}\ge 1$.
The following is due to Cassels \cite{cas}.

\begin{re} \label{cas} Let $X\in \Z[\zeta_m]$ where $m=pm'$ and  $p$ is a prime with $(p,m')=1$.
Write $X=\sum_{i=0}^{p-1}X_i \zeta_{p}^i$
with $X_i\in \Z[\zeta_{m'}]$. We have
$$(p-1){\cal M}(X) = \sum_{i<j}^{p-1} {\cal M}(X_i-X_j).$$
\end{re}

We denote the cyclic group of order $k$ by $C_k$.

\begin{thm} \label{weil}
Let $p$, $q$ be distinct primes and $r$ be a positive integer
with $\gcd(r,pq)=1$.  Let  $n=q^b$ where $b$ is a positive integer.
Suppose that $Y\bar{Y}=n$ for some $Y \in \Z[\zeta_{pr}]$.
If ${\rm ord}_p(q)$ is even and
$p>\max\{3^{n/2},n^2+n+1\}$, then $Y\eta\in \Z[\zeta_{r}]$ for some root of unity $\eta$.
\end{thm}
\begin{proof}
Write ${\rm ord}_p(q)=2f$, and define
$\sigma\in {\rm Gal}(\Q(\zeta_{pr})/\Q)$ by $\zeta_{pr}^{\sigma} = \zeta_{pr}^{q^f}$.
Note that  $\sigma$ fixes all prime ideals above $n$ in $\Z(\zeta_{pr})$ (see \cite[Thm.\ 2.1]{sc}, for instance).
Hence $Y^{\sigma}= Y\alpha$ for some unit $\alpha$.
Note $|\alpha|^2=Y^{\sigma}\overline{Y^{\sigma}}/(Y\overline{Y})=n^{\sigma}/n=1$.
Thus $\alpha$ is a root of unity, i.e., $\alpha=\pm \zeta_p^c\zeta_r^d$ for some integers $c,d$.
Let $e$ be an integer with $2e\equiv c \pmod{p}$.
Note that $\zeta_p^{\sigma}=\zeta_p^{-1}$ and thus
$$(Y\zeta_p^e)^{\sigma}
= Y(\pm \zeta_p^c\zeta_r^d)\zeta_p^{-e}
= Y (\pm \zeta_p^{c-e}\zeta_r^d)
= Y( \pm \zeta_p^{e}\zeta_r^d).$$
Let $Y_1 = Y\zeta_p^e$.
Then
\begin{equation}\label{eq2}
Y_1^{\sigma} =  Y_1(\pm \zeta_r^d).
\end{equation}
Write $Y_1=\sum_{i=0}^{p-1}a_i\zeta_{p}^i$
with $a_i\in \Z[\zeta_r]$.
Note ${\cal M}(Y_1)={\cal M}(Y)=n$.
By Result \ref{cas}, we have
\begin{equation}\label{eq4}
(p-1)n= (p-1){\cal M}(Y_1) = \sum_{i<j}^{p-1} {\cal M}(a_i-a_j).
\end{equation}
Let $t$ be the maximum number such that there are
distinct indices $i_1,\ldots,i_t$ with $a_{i_1}=\cdots=a_{i_t}$.
If $t\le p/2$, then, by (\ref{geoarith}), the right hand side of (\ref{eq4}) is at least $p^2/4$
and thus $4n>p$, contradicting the assumption $p>n^2+n+1$.
Hence $t\ge p/2$. Note that, by (\ref{geoarith}), the right hand side of (\ref{eq4}) is
at least  $(p-t)t$. Hence $(p-t)t \le n(p-1)$. Note that $(p-t)t$ is decreasing for $t\in [p/2,p]$
and that $(p-(p-n-1))(p-n-1) = n(p-1)+ p -n-1-n^2>n(p-1)$, as $p>n^2+n+1$.
Hence $t\ge p-n$. Recall that $a_{i_1}=\cdots=a_{i_t}$.
Note $Y_1 = \sum_{i=0}^{p-1}(a_i-a_{i_1})\zeta_{p}^i$. Thus,
writing $b_i = a_i-a_{i_1}$, we have $Y_1 = \sum_{i=0}^{p-1}b_i\zeta_{p}^i$
and $|\{i: b_i\neq 0 \}|\leq n$. If $|\{i: b_i\neq 0 \}|=1$, then the assertion of Theorem \ref{weil} holds.
Hence, to complete the proof, it suffices to show that
\begin{equation}\label{eq1}
2\le |\{i: b_i\neq 0 \}|\leq n
\end{equation}
leads to a contradiction.
Define $X=\sum b_ig^i$ where $g$ is a generator of $C_p$, the cyclic group of order $p$.
Let $K$ be the kernel of the ring homomorphism $\rho: \Z[\zeta_r][C_p]\to \Z[\zeta_{pr}]$
determined by $g\mapsto \zeta_p$. It is well known and straightforward
to verify that $K=\{AC_p: A\in \Z[\zeta_r][C_p]\}$.
Note that $\rho(X)=Y_1$, $\rho(X^{(-1)})=\overline{Y_1}$
and thus  $\rho(XX^{(-1)})=Y_1\overline{Y_1}=n$.  Hence
\begin{equation}\label{Y}
XX^{(-1)} = n + AC_p
\end{equation}
for some $ A\in \Z[\zeta_r][C_p]$.
Suppose $A\neq 0$.  Then there are at least $p-1$ nonzero
coefficients on the right hand side of (\ref{Y}).
On the  other hand, by (\ref{eq1}), there are at most $n(n-1)$ nonzero coefficients
on the left hand side of (\ref{Y}).
Hence $p \le n(n-1)+1 $ which contradicts the assumptions.
Thus $A=0$ and hence
\begin{equation}\label{Y1}
XX^{(-1)} = n.
\end{equation}
Recall that $X=\sum b_ig^i$. By (\ref{eq1}),
we can write $X=\sum_{j=1}^z b_{i_j} g^{i_j}$ with $2\le z\le n$ and
$b_{i_j}\neq 0$ for all $j$.
Note
\begin{equation}\label{Y3}
XX^{(-1)} = \sum_{k=0}^{p-1} \left( \sum_{i_r-i_s\equiv k \bmod p} b_{i_r}\overline{b_{i_s}}\right)g^k.
\end{equation}
Write  $S=\{i_1,\ldots,i_z\}$ and view $S$ as a subset of $\F_p$.
Note $2\le |S|\le  n$.
Suppose $S$ has a unique difference, say, $k=i_r-i_s\in \F_p$ is a unique difference in $S$.
Note that $k\neq 0$, as $|S|\ge 2$ and thus $0$ is not a unique difference in $S$.
In view of (\ref{Y3}),  the coefficient of $g^k$ in $XX^{(-1)}$ is nonzero.
But this contradicts $XX^{(-1)}=n$. Thus  $S$ has no unique difference.

\medskip

Using (\ref{eq2}), we get
$$Y_1^{\sigma} = \sum_{i=0}^{p-1}b_i^{\sigma}\zeta_{p}^{-i} =  \pm \zeta_r^dY_1
= \pm \zeta_r^d \sum_{i=0}^{p-1}b_i\zeta_{p}^{i}
$$
 Hence, for $i\neq 0$, we have $b_i\neq 0$ if
and only if $b_{p-i}\neq 0$. This implies that $S$ is symmetric.
As $S$ has no unique difference, we have $p\le 3^{|S|/2}$
by Theorem \ref{main}. As $|S|\le n$,
we conclude $p\le 3^{n/2}$,  contradicting the assumptions.
\end{proof}

\bigskip
\noindent
{\bf Example} Let $p=107$, $q=2$, 
and $n=8$ in Theorem \ref{weil}. Note that ${\rm ord}_{p}(2)$ is even
and that $p>\max\{3^{n/2},n^2+n+1\}$. Thus, for any prime odd prime $r$ and $Y\in \Z[\zeta_{107r}]$
with $|Y|^2=8$, we have $Y\eta\in \Z[\zeta_{r}]$ for some root of unity $\eta$.

\bigskip

\noindent
{\bf Acknowledgement}
We are grateful to the referees for their careful reading 
of the paper and suggestions which improved the exposition
of the paper.


\end{document}